\documentclass{amsart}
\usepackage{amsfonts,amsmath,amsthm,amscd,amssymb,latexsym,cite,verbatim,texdraw,floatflt,pb-diagram}

\usepackage{tikz} 
 \newtheorem{thm}{Theorem}[section]
 \newtheorem{cor}[thm]{Corollary}
  
 \newtheorem{lem}[thm]{Lemma}
 \newtheorem{prop}[thm]{Proposition}
 \theoremstyle{definition}
 \newtheorem{defn}[thm]{Definition}
 \theoremstyle{remark}
 \newtheorem{rem}[thm]{Remark}
 \newtheorem{ex}[thm]{Example}
 \numberwithin{equation}{section}

\begin{document}

%
%
%
%
%
%
%
%
%

\title[Fixed point results for multipoint Kannan-type mappings]{Fixed point results for multipoint Kannan-type mappings}


\author[Ravindra K. Bisht]{Ravindra K. Bisht}
\address{Department of Mathematics, National Defence Academy\\
 Khadakwasla-411023\\
Pune\\ India}
\email{ravindra.bisht@yahoo.com}

\author[Evgeniy Petrov]{Evgeniy Petrov}

\address{
Institute of Applied Mathematics and Mechanics\\
of the NAS of Ukraine\\
Batiuka str. 19\\
84116 Slovyansk\\
Ukraine}
\email{eugeniy.petrov@gmail.com}




\begin{abstract}
We introduce and study a new type of mappings in metric spaces termed $n$-point Kannan-type mappings. A fixed-point theorem is proved for these mappings. In general case such mappings are discontinuous in the domain but necessarily continuous at fixed points. Conditions under which usual Kannan mappings and mapping contracting the total pairwise distances between $n$ points are $n$-point Kannan-type mappings are found.
It is shown that additional conditions of asymptotic regularity and continuity allow to extend the value of the contraction coefficient in fixed-point theorems for $n$-point Kannan-type mappings.
\end{abstract}

\maketitle

\section{introduction}


In 1968, Kannan \cite{Ka68} established a fixed point theorem for a class of mappings that are neither contractions nor contractive. Subsequently, Subrahmanyam \cite{S75} demonstrated that Kannan’s theorem uniquely characterizes metric completeness, a property that Banach contractions do not share. Specifically, a metric space \(X\) is complete if and only if every Kannan-type mapping on \(X\) has a fixed point.

A mapping \( T \colon X \to X \) on a complete metric space \( (X, d) \) is defined as a Kannan mapping if it satisfies the inequality
\begin{align}
d(Tx, Ty) \leq \lambda \big( d(x, Tx) + d(y, Ty) \big),
\end{align}
for all \( x, y \in X \), where \( 0 \leq \lambda < \frac{1}{2} \). Every Kannan mapping is known to possess a unique fixed point \cite{Ka68}.

Building on Kannan’s foundational work, researchers have advanced fixed point theory by establishing results under diverse contractive conditions (see \cite{Rh77}, \cite{B23}, and references therein).
Typically, in fixed point theory, it is possible to distinguish at least three types of generalizations of the Kannan theorem: in the first case, the contractive nature of the mapping is weakened, see, e.g.~\cite{Is74,Re71,Re712,Bi72,Go19}; in the second case the topology is weakened, see, e.g.~\cite{Ak22,KI19,AA12,DLG12,Ho22,JS09}; the third case  involves theorems for some {K}annan-type multivalued mappings, see, e.g.~\cite{NS10, Um15,DD11}.


In 2023, Petrov \cite{P23} introduced a new class of mappings contracting the perimeters of triangles, formulated for three distinct points~-- over a century after the Banach contraction principle. This advancement extended Banach’s framework to a three-point setting, prompting further generalizations of contraction and contractive mappings, ultimately leading to fixed point theorems applicable to multi-point configurations \cite{PB23, PB23C, PP23C, PP24,BP24, JPS25}. In 2024, Petrov \cite{P24} introduced a new type of mapping in metric spaces, mappings contracting the total pairwise distances between \( n \) points, and proved an existence theorem for periodic points of such mappings. Unlike many other periodic point theorems, this theorem serves as a true generalization of Banach's contraction principle.

In what follows, we denote the cardinality of a set \( X \) by \( |X| \). Let \( (X, d) \) be a metric space with \( |X| \geqslant 2 \), and let \( x_1, x_2, \ldots, x_n \in X \), where \( n \geqslant 2 \). Denote by
\begin{equation}
S(x_1,x_2,\ldots,x_n) = \sum\limits_{1\leqslant i<j\leqslant n} d(x_i,x_j)
\end{equation}
the sum of all pairwise distances between the points in the set \( \{x_1, x_2, \ldots, x_n\} \).

\begin{defn}\label{d1} \cite{P24}
Let \( n \geqslant 2 \), and let \( (X, d) \) be a metric space with \( |X| \geqslant n \). A mapping \( T\colon X\to X \) is called a \emph{mapping contracting the total pairwise distances between \( n \) points} if there exists \( \alpha \in [0,1) \) such that the inequality
\begin{equation}\label{e1}
S(Tx_1,Tx_2,\ldots,Tx_n) \leqslant \alpha S(x_1,x_2,\ldots,x_n)
\end{equation}
holds for all \( n \) pairwise distinct points \( x_1, x_2, \ldots, x_n \in X \).
\end{defn}

\begin{rem}
Note that the requirement for \( x_1, x_2, \ldots, x_n \in X \) to be pairwise distinct is essential \cite{P24}.
\end{rem}


\section{General properties of $n$-point Kannan-type mappings}

Motivated by \cite{P24}, we introduce Kannan-type aggregate pairwise distance mappings between $n$ points as follows:

\begin{defn}\label{def:3.1}
Let $(X,d)$ be a metric space with $|X|\geqslant 2$. A mapping $T\colon X\rightarrow X$ shall be called an $n$-point Kannan-type mapping $(2\leqslant n\leqslant |X|, n\in\mathbb{N})$ in $X$ if there exists $\lambda\in[0, \frac{n-1}{n})$ such that
\begin{equation}\label{eq2.1}
S(Tx_1,Tx_2,\ldots,Tx_n)\leqslant\lambda(d(x_1,Tx_1)+d(x_2,Tx_2)+\cdots+d(x_n,Tx_n))
\end{equation}
for all pairwise distinct points $x_i\in X, i=1,2,\ldots,n, n\geqslant 2$.
\end{defn}

For \( n = 3 \) in Definition \ref{def:3.1}, the mapping \( T \) aligns with the generalized Kannan-type mapping defined by Petrov and Bisht in~\cite{PB23}. 

\begin{prop}\label{pro:3.1}
Let $(X,d)$ be a metric space with $|X| \geq 2$. Every Kannan contraction $T \colon X \to X$ with constant $\lambda \in \left[0, \frac{1}{n}\right)$ ($2\leqslant n \leqslant |X|$) is an $n$-point Kannan-type mapping.
\end{prop}

\begin{proof} Let $(X,d)$ be a metric space with $|X|\geqslant 2$, and $T\colon X\rightarrow X$ be a Kannan contraction. For any pairwise distinct points $x_1, x_2, \dots, x_n \in X$, the Kannan contraction gives
\[
d(Tx_i, Tx_j) \leqslant \lambda \bigl(d(x_i, Tx_i) + d(x_j, Tx_j)\bigr) \quad \text{for all } 1 \leqslant i < j \leqslant n.
\]
Summing over all pairs $(i,j)$ with $1 \leqslant i < j \leqslant n$, we obtain
\[
S(Tx_1, Tx_2, \dots, Tx_n) = \sum_{1 \leqslant i < j \leqslant n} d(Tx_i, Tx_j) \leqslant \lambda \sum_{1 \leqslant i < j \leqslant n} \bigl(d(x_i, Tx_i) + d(x_j, Tx_j)\bigr).
\]

Observe that for each fixed $i$, the term $d(x_i, Tx_i)$ appears exactly $n-1$ times in the double sum (once for each $j \neq i$). Therefore,
\[
\sum_{1 \leqslant i < j \leqslant n} \bigl(d(x_i, Tx_i) + d(x_j, Tx_j)\bigr) = (n-1) \sum_{i=1}^n d(x_i, Tx_i).
\]

Substituting back, we obtain
\[
S(Tx_1, Tx_2, \dots, Tx_n) \leqslant \lambda (n-1) \sum_{i=1}^n d(x_i, Tx_i).
\]
Hence, $T$ is an $n$-point Kannan-type mapping.
\end{proof}

\begin{lem}\label{lem:3.2} Let $(X,d)$ be a metric space with $|X|\geqslant 2$, and let $T\colon X\rightarrow X$ be an $n$-point Kannan-type mapping $(2\leqslant n\leqslant |X|, n\in\mathbb{N})$. If $z$ is an accumulation point of $X$ and $T$ is continuous, then the inequality
\begin{equation}\label{eq:3.11}
(n-1)d(Tz,Ty)\leqslant\lambda((n-1)d(z,Tz)+d(y,Ty)).
\end{equation}
holds for all $y\in X$.
\end{lem}

\begin{proof}
Given any accumulation point $z\in X$, and any $y\in X$. If $z=y$, then \eqref{eq:3.11} holds trivially.\\
Assume that $z\neq y$, since $z$ is an accumulation point of $X$, there exists a sequence $x_k\rightarrow z$ such that $x_k\neq z, x_k\neq y$ for all $k$ and all $x_k$ are pairwise distinct. Hence, by \eqref{eq2.1}, we have
\begin{align}\label{eq:3.3}
&S(Tz,Ty,Tx_1,Tx_{2},\ldots,Tx_{n-2})\notag\\
&\leqslant\lambda(d(z,Tz)+d(y,Ty)+d(x_{k+1},Tx_{k+1})+\cdots+d(x_{k+n-2},Tx_{k+n-2})),
\end{align}
for all $k\in\mathbb{N}$.

Since $d(x_k,z)\rightarrow 0$,
due to the continuity of $T$, we have $Tx_k\to Tz$, $d(Ty,Tx_k)\rightarrow d(Ty,Tz)$ and $d(x_k,Tx_k)\rightarrow d(z,Tz)$. Letting $n\rightarrow\infty$ in \eqref{eq:3.3}, we have \eqref{eq:3.11}.
\end{proof}

\begin{prop}\label{pro:3.3} Let $(X,d)$ be a metric space with $|X|\geqslant 2$, and $T\colon  X\rightarrow X$ be a continuous $n$-point Kannan-type mapping $(2\leqslant n\leqslant |X|, n\in\mathbb{N})$. Suppose that all points of $X$ are accumulation points. Then $T$ is a Kannan contraction.
\end{prop}
\begin{proof}
By Lemma~\ref{lem:3.2}, in addition to~\eqref{eq:3.11} we also get the inequality
\begin{eqnarray}\label{eq:1.15}
(n-1)d(Ty,Tz)\leqslant\lambda((n-1)d(y,Ty)+d(z,Tz))
\end{eqnarray}
for all $y,z\in X$. Adding the left and right sides of inequalities~\eqref{eq:3.11} and ~\eqref{eq:1.15}, we get
\begin{align*}
d(Tz,Ty)\leqslant\frac{n\lambda}{2(n-1)}(d(z,Tz)+d(y,Ty)).
\end{align*}
Since $\lambda\in[0,\frac{n-1}{n})$, we have $\frac{n\lambda}{2(n-1)}\in [0, \frac{1}{2})$, which completes the proof.
\end{proof}

\begin{prop}\label{pro:3.4}
Let $(X,d)$ be a metric space with $|X|\geqslant 2$, and $T\colon X\rightarrow X$ be a mapping contracting the total pairwise distances between $n$ points $(2\leqslant n\leqslant |X|, n\in\mathbb{N})$ with $0\leqslant\alpha<\frac{1}{n+1}$.
Then $T$ is an $n$-point Kannan-type mapping.
\end{prop}

\begin{proof}
By the supposition~\eqref{eq2.1} holds for all pairwise distinct points $x_i\in X, i=1,2,\ldots,n$. 
Using the triangle inequalities for each pair
$$d(x_i,x_j)\leqslant d(x_i,Tx_i)+d(Tx_i,Tx_j)+d(Tx_j,x_j),$$ we get
\[
S(x_1,\ldots,x_n) \leqslant \sum_{1\leqslant i<j\leqslant n} \Big(d(x_i,Tx_i) + d(Tx_i,Tx_j) + d(Tx_j,x_j)\Big).
\]
Thus
\begin{equation}\label{66}
    S(x_1,\ldots,x_n) \leqslant (n-1)\sum_{i=1}^n d(x_i,Tx_i) + S(Tx_1,\ldots,Tx_n).
    \end{equation}
Substituting \eqref{66} into \eqref{eq2.1}, we obtain
    \[
    S(Tx_1,\ldots,Tx_n) \leqslant \alpha\left((n-1)\sum_{i=1}^n d(x_i,Tx_i) + S(Tx_1,\ldots,Tx_n)\right).
    \]
This implies
 \begin{align*}
    S(Tx_1,\ldots,Tx_n) &\leqslant \frac{\alpha(n-1)}{1-\alpha}\sum_{i=1}^n d(x_i,Tx_i).
    \end{align*}
Let $\lambda = \frac{\alpha(n-1)}{1-\alpha}$. Since $0\leqslant \alpha <\frac{1}{n+1}$, we get $0\leqslant \lambda <\frac{n-1}{n}$.
Hence, $T$ is an $n$-point Kannan-type mapping.
\end{proof}

\section{The main results}

The following lemma describes a sufficiently general condition guaranteeing that a sequence in a metric space is a Cauchy sequence.

\begin{lem}\label{l26}
Let $(X,d)$ be a metric space, let $n\geqslant 2$ $n\in \mathbb N$ and let $(x_m)$, $m=0,1,2,...$ be a sequence in $(X,d)$. Set $p_m=d(x_m, x_{m+1})$. Suppose that for every $m\geqslant n-1$ the inequality
\begin{equation}\label{e2}
p_m\leqslant \rho\max\limits_{m-n+1 \leqslant i\leqslant m-1}p_i
\end{equation}
holds for some $0\leqslant \rho <1$, i.e. every distance $p_m$ is not greater then the maximal distance of $n-1$ previous distances multiplied by some real number \mbox{$\rho\in [0,1)$}. Then $(x_m)$ is a Cauchy sequence.
\end{lem}
\begin{proof}
Let $P = \max\{p_0, \dots, p_{n-2}\}$. Hence, by~(\ref{e2}) we obtain
$$
p_0\leqslant P, \quad p_1\leqslant P,\ldots,\quad p_{n-2}\leqslant P,
$$
$$
p_{n-1}\leqslant \rho P, \quad p_{n}\leqslant \rho P,\ldots,\quad p_{2n-3}\leqslant \rho P,
$$
$$
p_{2n-2}\leqslant \rho^2 P, \quad p_{2n-1}\leqslant \rho^2 P,\ldots,\quad p_{3n-4}\leqslant \rho^2 P,
$$
$$
p_{3n-3}\leqslant \rho^3 P, \quad p_{3n-2}\leqslant \rho^3 P,\ldots,\quad p_{4n-5}\leqslant \rho^3 P,
$$
$$
\cdots .
$$
Since $\rho<1$, it is clear that the inequalities
$$
p_{n-1}\leqslant \rho^{\frac{1}{n-1}} P, \quad p_{n}\leqslant \rho^{\frac{2}{n-1}} P,\ldots,\quad p_{2n-3}\leqslant \rho P,
$$
$$
p_{2n-2}\leqslant \rho^{1+\frac{1}{n-1}} P, \quad p_{2n-1}\leqslant \rho^{1+\frac{2}{n-1}} P,\ldots,\quad p_{3n-4}\leqslant \rho^2 P,
$$
$$
p_{3n-3}\leqslant \rho^{2+\frac{1}{n-1}} P, \quad p_{3n-2}\leqslant \rho^{2+\frac{2}{n-1}} P,\ldots,\quad p_{4n-5}\leqslant \rho^3 P,
$$
$$
\cdots .
$$
also hold, i.e.,
\begin{equation}\label{e3}
  p_m\leqslant \rho^{\frac{m-n+2}{n-1}}P
\end{equation}
for $m=n-1, n, n+1,...$.

Let $k\in \mathbb N$, $k\geqslant 2$. By the triangle inequality, for $m\geqslant n-1$ we have
$$
d(x_m,x_{m+k})\leqslant d(x_m,x_{m+1})+d(x_{m+1},x_{m+2})+\cdots+d(x_{m+k-1},x_{m+k})
$$
$$
=p_m+p_{m+1}+\cdots+p_{m+k-1}\leqslant
P(\rho^{\frac{m-n+2}{n-1}}+\rho^{\frac{m-n+3}{n-1}}+\cdots
+\rho^{\frac{m+k-n+1}{n-1}})
$$
$$
=P\rho^{\frac{m-n+2}{n-1}}(1+\rho^{\frac{1}{n-1}}+\cdots
+\rho^{\frac{k-1}{n-1}})=P\rho^{\frac{m-n+2}{n-1}}
\frac{1-\rho^{\frac{k}{n-1}}}{1-\rho^\frac{1}{n-1}}.
$$
Since by supposition $0 \leqslant \rho <1$, it follows that
$0 \leqslant \rho^{\frac{k}{n-1}} <1$ and
$$
d(x_m,x_{m+k})\leqslant P\rho^{\frac{m-n+2}{n-1}}
\frac{1}{1-\rho^\frac{1}{n-1}}.
$$
Hence, $d(x_m,x_{m+k})\to 0$ as $m\to \infty$ for every $k\geqslant 2$. For $k=1$ this follows from~(\ref{e3}). Thus, $(x_m)$ is a Cauchy sequence.
\end{proof}

Let $T$ be a mapping on the metric space $X$. Recall, that a point $x\in X$ is called a \emph{periodic point of period $n$} if $T^n(x) = x$. The least positive integer $n$ for which $T^n(x) = x$ is called the prime period of $x$, see, e.g.,~\cite[p.~18]{De22}.

The central result of this paper is presented in the following theorem.

\begin{thm}\label{thm1}
Let $(X,d)$ be a complete metric space with $|X| \geq n \geq 2$. Let $T\colon  X \rightarrow X$ be an $n$-point Kannan-type mapping. If $T$ does not have periodic points of prime periods $2, 3, \dots, n-1$, then $T$ has a fixed point in $X$. Moreover, $T$ can admit at most $n-1$ fixed points.
\end{thm}

\begin{proof}
Let $T\colon  X \rightarrow X$ be an $n$-point Kannan-type mapping that does not have periodic points of prime periods $2, 3, \dots, n-1$.

For any $x_0 \in X$, define the sequence $\{x_m\}$ by $x_m = Tx_{m-1}$ for $m \in \mathbb{N}$. If $x_m$ is a fixed point for some $m$, we are done. Otherwise, assume $x_m \neq Tx_m$ for all $m$, which implies $x_m \neq x_{m+1}$ for all $m$. Since $T$ has no periodic points of prime periods $2, \dots, n-1$, every $n$ consecutive terms of $(x_m)$ are distinct.

For any $m \geq 0$, (\ref{eq2.1}) implies
\begin{align*}
S(x_{m+1}, \dots, x_{m+n}) &\leqslant \lambda \sum_{i=1}^{n} d(x_{m+i}, x_{m+i+1}).
\end{align*}
Expanding $S$ and isolating $d(x_{m+n-1}, x_{m+n})$, we get
\begin{align*}
(1-\lambda)d(x_{m+n-1}, x_{m+n}) \leqslant &\lambda \sum_{i=1}^{n-2} d(x_{m+i}, x_{m+i+1})+d(x_{m+n}, x_{m+n+1})
 \\
 -&\sum_{\substack{1 \leqslant i < j \leqslant n \\ (i,j) \neq (n-1,n)}} d(x_{m+i}, x_{m+j}).
\end{align*}
Using the evident inequalities
$$
-d(x_{m+n-1},x_k)-d(x_k,x_{m+n})\leqslant -d(x_{m+n-1},x_{m+n})
$$
for $k=m+1,\ldots, m+n-2$ and omitting unnecessary negative terms in the sum above, we get
\begin{align*}
(n-1-\lambda)d(x_{m+n-1}, x_{m+n}) &\leqslant \lambda \sum_{i=0}^{n-2} d(x_{m+i}, x_{m+i+1}).
\end{align*}
Rearranging, we get
\begin{align*}
d(x_{m+n-1}, x_{m+n}) &\leqslant \frac{\lambda}{(n-1-\lambda)} \sum_{i=0}^{n-2} d(x_{m+i}, x_{m+i+1}),
\end{align*}
and
\begin{align*}
    & d(x_{m+n-1}, x_{m+n}) \leqslant \frac{(n-1)\lambda}{(n-1-\lambda)} \max_{0 \leqslant i \leqslant n-2} d(x_{m+i}, x_{m+i+1}).
\end{align*}

Let $\rho = \frac{(n-1)\lambda}{n-1-\lambda}$. Since $\lambda < \frac{n-1}{n}$, we have $\rho < 1$ and
\begin{equation*}
d(x_{m+n-1}, x_{m+n}) \leqslant \rho \max_{0 \leqslant i \leqslant n-2} d(x_{m+i}, x_{m+i+1}).
\end{equation*}
By Lemma~\ref{l26},
the sequence \( (x_m) \) is Cauchy and, by the completeness of \( X \), it converges to some \( w \in X \).

Recall that any $n$ consecutive elements of the sequence $(x_m)$ are pairwise distinct. If
$w\neq x_m$ for all $m\in \{1, 2, ...\}$, then inequality ~(\ref{eq2.1}) holds for $n$ pairwise distinct points $x_{m-1},x_m,...,x_{m+n-3},w$. Suppose that there exists the smallest possible $k\in \{1, 2, ...\}$ such that $w = x_k$. Let $t > k$ be such that $w=x_t$. Then the sequence $(x_m$) is cyclic starting from $k$ and cannot be a Cauchy sequence. Hence, the points $x_{m-1},x_m,...,x_{m+n-3},w$ are pairwise distinct at least when $m-1>k$.
Further, we have
\begin{align} \label{33}
d(w, Tw) &\leqslant d(w, x_m) + d(Tx_{m-1}, Tw) \\ \notag
&\leqslant d(w, x_m) + S(Tx_{m-1}, Tx_m,...,Tx_{m+n-3}, Tw)\\ \notag
&\leqslant d(w, x_m) + \lambda\left( d(w, Tw)+\sum_{j=m-1}^{m+n-3} d(x_j, x_{j+1})\right).
\end{align}
Rearranging this inequality, we get
\begin{align}\label{34}
d(w, Tw) \leqslant \frac{1}{1-\lambda}
\left( d(w, x_m) + \lambda\sum_{j=m-1}^{m+n-3} d(x_j, x_{j+1}) \right)\to 0,
\end{align}
as $m\to \infty$. Hence, $Tw = w$.

If \( T \) has \( n \) distinct fixed points \( w_1, \dots, w_n \), then by applying ~\eqref{eq2.1}, we obtain
\begin{equation*}
S(w_1, \dots, w_n) = S(Tw_1, \dots, Tw_n) \leqslant \lambda \sum_{i=1}^{n} d(w_i, Tw_i) = 0.
\end{equation*}
Since \( S(w_1, \dots, w_n) = 0 \) contradicts the distinctness of the fixed points, it follows that \( T \) can have at most \( n-1 \) fixed points.

\end{proof}

\begin{prop}\label{pro:3.5} Suppose that under the assumption of Theorem \ref{thm1}, the mapping $T$ has a fixed point $w$ that acts as the limit for a specific iteration sequence $\{x_i\}_0^\infty$ defined by $x_i=Tx_{i-1}, i\in\mathbb{N}$ with $w\neq x_i$ for all $i\in\mathbb{N}\cup\{0\}$, then $w$ is the unique fixed point of $T$.
\end{prop}
\begin{proof} Suppose that $z$ is another fixed point of $T$. Then $z\neq x_m$ for all $m\in\mathbb{N}\cup\{0\}$, otherwise, we have $w=z$. Therefore, $w,z,x_m$ are all distinct for all $m\in\mathbb{N}\cup\{0\}$.
Thus, by~(\ref{eq2.1}) for all $m\in\mathbb{N}\cup\{0\}$, we have
\begin{align*}
&S(Tw,Tz,Tx_m,...,Tx_{m+n-3})\\
&\leqslant\lambda(d(w,Tw)+d(z,Tz)+d(x_m,Tx_m)+\cdots+d(x_{m+n-3},Tx_{m+n-3})),
\end{align*}
which implies that
\begin{align*}
S(w,z,x_{m+1},...,x_{m+n-2})
\leqslant\lambda(d(x_m,x_{m+1})+\cdots+d(x_{m+n-3},x_{m+n-2})).
\end{align*}
Letting $m\rightarrow \infty$ in the above inequality, we have $(n-1)d(w,z)\leqslant 0$, which contradicts to the fact that $w\neq z$. Therefore, $T$ has a unique fixed point.
\end{proof}

It is not hard to see that Theorem \ref{thm1} admits the following reformulation.
\begin{thm}\label{thm2}
Let $(X,d)$ be a complete metric space with $|X| \geq n \geq 2$. Let $T\colon  X \rightarrow X$ be an $n$-point Kannan-type mapping. Then $T$ has a periodic point of prime period $p$, $p\in \{1,2, 3, \dots, n-1\}$.
\end{thm}
Recall that the orbit of the point $x$ under the mapping $T$ is the set $O_T(x)=\{x, Tx, T^2x,\ldots\}$.
\begin{rem}
Let $n\geqslant 3$. Suppose that in Theorem~\ref{thm2}, there exist two periodic points $x_1$ and $x_2$ with prime periods $p_1$ and $p_2$, respectively, $p_1,p_2\neq 2$, such that their orbits are distinct, and $p_1+p_2=n$. Then,
$$
O_T(x_1)=\{x_1, Tx_1,\ldots, T^{p_1-1}x_1\} \text{ and } O_T(x_2)=\{x_2, Tx_2,\ldots, T^{p_2-1}x_2\}.
$$
Substituting the pairwise distinct points of the orbits $O_T(x_1)$ and $O_T(x_2)$ into~(\ref{eq2.1}), we get

$$
S(Tx_1,T^2x_1,\ldots,T^{p_1-1}x_1,x_1,
Tx_2,T^2x_2,\ldots,T^{p_1-1}x_2,x_2)
$$
$$
\leqslant
\lambda
(d(x_1,Tx_1)+d(Tx_1,T^2x_1)+\cdots+d(T^{p_1-1}x_1,x_1))
$$
$$
+(d(x_2,Tx_2)+d(Tx_2,T^2x_2)+\cdots+d(T^{p_2-1}x_2,x_2)),
$$

which is a contradiction, since all the distances on the right side of this inequality are already included on the left side, and $\lambda <1$.

Analogously, the case when  $x_1,x_2,\ldots,x_k$, $k=1,2,\ldots$, are periodic points of the mapping $T$, each having a distinct orbit and corresponding prime periods $p_1,p_2,\ldots,p_k$, respectively with $p_i\neq 2$ for all $i\in\{1,...,k\}$, and $p_1+p_2+\cdots+p_k=n$, is also impossible.

In particular, this means that $T$ can admit at most $n-1$ fixed points and  cannot have a periodic point of prime period $n$. The case when $p_i=2$ for some $i\in\{1,...,k\}$ is also impossible. The proof is connected with simple analysis of triangle inequalities and is left to the reader.


\end{rem}
Numerous examples of $3$-point Kannan-type mappings are presented in~\cite{PB23}. To illustrate Theorem \ref{thm1} we now provide an example of $n$-point Kannan-type mapping, which is not an $(n-1)$-point Kannan-type mapping.

\begin{ex}
Let $(X,d)$ be a metric space with $X = \{x_1, \ldots, x_n\}$, $n \geq 3$, where the metric is defined by:
\[
d(x_i, x_j) =
\begin{cases}
0 & \text{if } i = j, \\
1 & \text{if } 1 \leqslant i,j \leqslant n-1 \text{ and } i \neq j, \\
M & \text{if } \max\{i,j\} = n \text{ and } i \neq j,
\end{cases}
\]
where $M > 1$ is a constant. The reader can easily verify that $(X,d)$ is a metric space and even ultrametric space. Define the mapping $T\colon X \to X$ as
\[
Tx_i =
\begin{cases}
x_{i+1} & \text{for } 1 \leqslant i \leqslant n-2, \\
x_{n-1} & \text{for } i = n-1, \\
x_1 & \text{for } i = n.
\end{cases}
\]
Here,
$$
S(Tx_1, \ldots, Tx_n)=
S(x_2,x_3, \ldots, x_{n-1}, x_{n-1}, x_1)=\binom{n}{2}-1=\frac{n(n-1)}{2}-1.
$$
and
\begin{align*}
\sum_{i=1}^n d(x_i, Tx_i) &= \sum_{i=1}^{n-2} d(x_i, x_{i+1}) + d(x_n, x_1)= n-2 + M.
\end{align*}
The ratio is
\begin{align*}
R = \frac{S(Tx_1, \ldots, Tx_n)}{\sum_{i=1}^n d(x_i, Tx_i)}
= \frac{\frac{n(n-1)}{2} -1}{n-2 + M}
\leqslant \frac{n-1}{n}.
\end{align*}
It is clear that this  inequality holds for sufficiently large $M$.

The point $x_{n-1}$ is the unique fixed point since $Tx_{n-1} = x_{n-1}$. Thus, $T$ is an $n$-point Kannan-type mapping with $\lambda \in [R, \frac{n-1}{n})$.

Let us show that for $n\geqslant 4$, the mapping $T$ is not an $(n-1)$-point Kannan-type mapping. Indeed,
$$
S(Tx_1, \ldots, Tx_{n-1})=
S(x_2,x_3, \ldots, x_{n-1}, x_{n-1})=\binom{n-1}{2}-1=\frac{(n-1)(n-2)}{2}-1.
$$
\begin{align*}
\sum_{i=1}^{n-1} d(x_i, Tx_i) = \sum_{i=1}^{n-2} d(x_i, x_{i+1})= n-2.
\end{align*}
It is easy to verify that the inequality
$$
\frac{(n-1)(n-2)}{2}-1 \geqslant n-2,
$$
holds for every $n\geqslant 4$, which makes inequality~(\ref{eq2.1}) impossible for $n-1$ points. This completes the proof.
\end{ex}

\begin{rem}
If $x$ is an isolated point of a metric space, then clearly any mapping is continuous at $x$. Suppose in the definition of sequential continuity in a metric space at an accumulation point $x$, we consider sequences $(x_m)$ converging to $x$, where all terms $(x_m)$ are distinct from each other (i.e., $x_i\neq x_j$ for $i\neq j$) and $x_m\neq x^*$ for all $m$. It is not hard to see that this restricted version is still equivalent to the usual definition of continuity.
\end{rem}

It is well known that standard Kannan mappings are, in general, discontinuous, but necessarily continuous at their fixed points. Examples of discontinuous $3-$point Kannan-type mappings can be found in~\cite{PB23}.

\begin{prop}
$n$-point Kannan type mappings are continuous at fixed points.
\end{prop}

\begin{proof}
Let $(X,d)$ be a metric space with $|X|\geqslant 2$, $T\colon X\to X$ be an $n$-point Kannan-type mapping and $x^*$ be a fixed point of $T$. If $x^*$ is isolated, then the assertion is evident.  Let $(x_m)$ be a sequence such that $x_m\to x^*$, $x_{i}\neq x_{j}$ for all $i\neq j$ and $x_{m}\neq x^*$ for all $m$. Let us show that $Tx_m\to Tx^*$. Since all the $n$ points $x^*,x_m, x_{m+1},\ldots,x_{m+n-2}$ are pairwise distinct, by~(\ref{eq2.1}) we have
  \begin{multline*}
   S(Tx^*,Tx_m, Tx_{m+1},\ldots,Tx_{m+n-2}) \\ \leqslant \lambda (d(x^*,Tx^*)+d(x_m,Tx_m)+d(x_{m+1},Tx_{m+1})+\cdots+d(x_{m+n-2},Tx_{m+n-2})).
  \end{multline*}
Hence,
  \begin{equation*}
   d(Tx^*,Tx_m) \leqslant \lambda (d(x_m,x_{m+1})+\cdots+d(x_{m+n-2},x_{m+n-1})) \to 0
  \end{equation*}
  as $m\to \infty$
since, clearly, $d(x_m,x^*)\to 0$ implies $d(x_{m},x_{m+1})\to 0$. This completes the proof.
\end{proof}

We now present an example illustrating that an $n-$point Kannan-type mapping can be discontinuous at every point of the domain except at its fixed point, which is the only point of continuity.

\begin{ex}
Let \( X = [0,2] \), \( d \) be the usual Euclidean metric on $X$ and let \( n \geq 2 \) be a fixed integer. Define a mapping \( T \colon X \to X \) by
\[
T(x) =
\begin{cases}
\dfrac{x}{n+2}, & \text{if } x \text{ is rational}, \\
0, & \text{if } x \text{ is irrational}.
\end{cases}
\]
Then \( T \) satisfies the Kannan condition
\[
d(Tx, Ty) \leq \dfrac{1}{n+1} \left( d(x, Tx) + d(y, Ty) \right)
\]
for all \( x, y \in X \). Therefore, by Proposition~\ref{pro:3.1}, \( T \) qualifies as an $n-$point Kannan-type mapping for a fixed integer $n\geq 2$. It is important to note that a $2-$point Kannan-type mapping is essentially the classical Kannan condition. Also, observe that \( T \) is discontinuous at every point of \( X \) except at \( x = 0 \), which is the unique fixed point of \( T \).
\end{ex}

\section{$n$-point $G$-Kannan-type mapping, asymptotic regularity and multiple fixed points}

Asymptotic regularity constitutes a fundamental condition that enhances the applicability of numerous contractive mappings, thereby enabling fixed point theorems to accommodate a broader class of mappings.

Let $(X, d)$ be a metric space. A mapping $T\colon X \to X$ is called asymptotically regular \cite{BP66} if it satisfies
\begin{equation}\label{ar}
\lim_{n \to \infty} d(T^{n}x, T^{n+1}x) = 0
\end{equation}
for all $x \in X$.

\begin{rem}\label{r32}
Let \((X, d)\) be a metric space, and let \(T \colon X \to X\) be a self-mapping. Suppose we generate a sequence \((x_n)\) starting from \(x_0 \in X\) via the iteration \(x_{n}=T^nx_0\) for all \(n \geq 1\). Assume that \(T\) is asymptotically regular at each point of \(X\), and further, that \((x_n)\) does not converge to a fixed point of \(T\). Then the sequence \((x_n)\) contains no repetitions, that is, \(x_i \ne x_j\) for all \(i \ne j\). To see this, note that if two terms of the sequence were equal, then the iterates would repeat in a cycle, making the sequence eventually periodic. However, such periodicity would contradict the asymptotic regularity of \(T\), which requires that the distances \(d(x_n, x_{n+1})\) vanish in the limit.
\end{rem}

We now introduce a more general class of mappings, termed \emph{$n$-point $G$-Kannan-type mappings}, which encompass the class of \emph{$n$-point Kannan-type mappings} as a special case.

First, we define the class \(\mathcal{G}\) of functions $G\colon \mathbb{R}_+^n \to \mathbb{R}_+$ satisfying the following properties

\begin{enumerate}
    \item[(i)] \(G(0, 0, \ldots, 0) = 0\);
    \item[(ii)] \(G\) is continuous at the point \((0, 0, \ldots, 0)\).
\end{enumerate}

\begin{defn}\label{def:4.1}
Let \((X, d)\) be a metric space with \(|X| \geqslant 2\). A mapping \(T\colon X \to X\) is said to be an \emph{$n$-point $G$-Kannan-type mapping} in \(X\) (for \(2 \leqslant n \leqslant |X|\), \(n \in \mathbb{N}\)) if there exists a function \(G \in \mathcal{G}\) such that the following inequality holds
\begin{equation}\label{eq4.1}
S(Tx_1, Tx_2, \ldots, Tx_n) \leqslant G\bigl(d(x_1, Tx_1), d(x_2, Tx_2), \ldots, d(x_n, Tx_n)\bigr)
\end{equation}
for all collections of pairwise distinct points \(x_i \in X\), \(i = 1, 2, \ldots, n\).
\end{defn}

\begin{thm}\label{thm7}
Let $(X,d)$ be a complete metric space with $|X| \geq n \geq 2$. Let $T\colon  X \rightarrow X$ be a continuous, asymptotically regular $n$-point $G$-Kannan-type mapping.  Then $T$ has a fixed point in $X$. Moreover, $T$ can admit at most $n-1$ fixed points.
\end{thm}

\begin{proof}
Let $x_0 \in X$, and define $T(x_0) = x_1$, $T(x_1) = x_2$, and so on. Suppose that the sequence $(x_m)$ does not possess a fixed point of $T$. Let us prove that $(x_m)$ is a Cauchy sequence. It is sufficient to show that $d(x_m,x_{m+p})\to 0$ as $m\to \infty$ for all $p>0$.
If $p=1$, then this follows from the definition of asymptotic regularity. Let $p\geqslant 2$. By Remark~\ref{r32} the points $x_{m}$, $x_{m+p-1}$, $x_{m+p}$,\ldots,$x_{m+p+n-3}$ are pairwise distinct. Using the triangle inequality, inequality~(\ref{eq4.1}), asymptotic regularity, and properties (i) and (ii), we obtain
\begin{align*}
&d(x_m,x_{m+p})
\leqslant d(x_m,x_{m+1})+d(x_{m+1},x_{m+p})\\
&\leqslant d(x_m,x_{m+1})+
S(x_{m+1},x_{m+p},x_{m+p+1},\ldots,x_{m+p+n-2})\\
&= d(x_m,x_{m+1})+
S(Tx_{m},Tx_{m+p-1},Tx_{m+p},\ldots,Tx_{m+p+n-3})\\
&\leqslant d(x_m,x_{m+1})+
G(d(x_{m},Tx_{m}),d(x_{m+p-1},Tx_{m+p-1}),d(x_{m+p},Tx_{m+p}),\ldots \\
&\ldots,d(x_{m+p+n-3},Tx_{m+p+n-3}))\\
&= d(x_m,x_{m+1})+
G(d(x_{m},x_{m+1}),d(x_{m+p-1},x_{m+p}),d(x_{m+p},x_{m+p+1}),\ldots \\
&\ldots,d(x_{m+p+n-3},x_{m+p+n-2})) \to 0 \text{ as } m\to \infty.
\end{align*}
Thus, $(x_{m})$ is a Cauchy sequence.
Since \(X\) is complete, there exists \(x^* \in X\) such that \(x_m \to x^*\).
Continuity of \(T\) implies
\[
Tx^* = T\left(\lim_{m \to \infty} x_m\right) = \lim_{m \to \infty} T x_m = \lim_{m \to \infty} x_{m+1} = x^*.
\]
The rest of the proof follows from reasoning similar to the last paragraph of the proof of Theorem~\ref{thm1}.
\end{proof}

\begin{cor}\label{cor12}
Let \((X,d)\) be a complete metric space with \(|X| \geq n \geq 2\). Let \(T\colon X \rightarrow X\) be a continuous and asymptotically regular mapping. Suppose \(T\) satisfies the \(n\)-point \(\mathcal{B}\)-Kannan-type mapping: for all collections of pairwise distinct points \(x_1, \ldots, x_n \in X\),
\begin{equation}\label{eq4.2}
S(Tx_1, \ldots, Tx_n) \leq \sum_{i=1}^{n} \beta_i(d(x_i, Tx_i))d(x_i, Tx_i),
\end{equation}
where each \(\beta_i \in \mathcal{B}\), and \(\mathcal{B}\) denotes the class of functions \(\beta\colon [0, \infty) \to [0, \infty)\) satisfying
\[
\limsup_{t \to 0^+} \beta(t) < \infty.
\]
Then $T$ has a fixed point in $X$. Moreover, $T$ can admit at most $n-1$ fixed points.
\end{cor}
\begin{proof}
Define
\[
G(x_1, x_2, \ldots, x_n) = \sum_{i=1}^{n} \beta_i(x_i)\, x_i.
\]
Clearly, \(G(0, 0, \ldots, 0) = 0\). Moreover, since each \(\beta_i \in \mathcal{B}\), we have \(\limsup\limits_{t \to 0^+} \beta_i(t) < \infty\) for \(i = 1, 2, \ldots, n\), which implies
\[
\lim_{x_1, x_2, \ldots, x_n \to 0} G(x_1, x_2, \ldots, x_n) = 0.
\]
Hence, \(G \in \mathcal{G}\), and the result follows from Theorem~\ref{thm7}.
\end{proof}

By setting \(\beta_1(t) = \beta_2(t) = \cdots = \beta_n(t) = \lambda\) with \(\lambda \geq 0\) in (\ref{eq4.2}),  we obtain an $n$-point Kannan-type mapping with coefficient \(\lambda \in [0,\infty)\).

Hence, we immediately obtain the following result.

\begin{cor}\label{cor16}
Let \((X,d)\) be a complete metric space with \(|X| \geq n \geq 2\). Let \(T\colon X \to X\) be a continuous and asymptotically regular generalized $n$-point Kannan-type mapping with the coefficient \(\lambda \in [0,\infty)\). Then $T$ has a fixed point in $X$. Moreover, $T$ can admit at most $n-1$ fixed points.
\end{cor}

In the following theorem, we demonstrate that omitting the continuity assumption allows us to restrict the constant $\lambda$ to the interval $[0,1)$ instead of $[0, \infty)$.

\begin{thm}\label{th16}
Let \((X,d)\) be a complete metric space with \(|X| \geq n \geq 2\). Let \(T\colon X \to X\) be an asymptotically regular $n$-point Kannan-type mapping with coefficient \(\lambda \in [0, 1)\).
Then $T$ has a fixed point in $X$. Moreover, $T$ can admit at most $n-1$ fixed points.
\end{thm}

\begin{proof}
The steps of the proof are analogous to those in Theorem \ref{thm7} up to the point where the sequence is shown to be Cauchy, followed by equations (\ref{33}) and (\ref{34}) as given in Theorem \ref{thm1}.
\end{proof}

\begin{rem}
Theorem \ref{thm7} remains valid if the continuity of the mapping
$T$ is relaxed to weaker continuity notions, which have been thoroughly studied in \cite{B23}.
\end{rem}


\noindent\textbf{Data availability} Data sharing is not applicable to this article as no data sets were generated or analyzed during the current study.

\noindent\textbf{Declarations}\\

\noindent\textbf{Conﬂict of interest} The authors declare that they have no conﬂicts of interest to disclose.


\end{document}